\documentclass[11pt]{amsart}
\usepackage[utf8]{inputenc}
\usepackage{amsmath}
\usepackage{amssymb}
\usepackage{amsthm}
\usepackage{enumerate}
\usepackage{graphicx}
\usepackage{import}
\usepackage{xifthen}
\usepackage{pdfpages}
\usepackage{tikz} 
\usepackage{tikz-cd} 
\usepackage{mathabx}
\usepackage{hyperref}
\usepackage{thmtools, thm-restate}

\usetikzlibrary{shapes.geometric}
\usetikzlibrary{graphs,graphs.standard}
\usetikzlibrary{fit}
\usetikzlibrary{calc}

\usepackage[style=alphabetic,giveninits]{biblatex}
\theoremstyle{plain}
\newtheorem{theorem}{Theorem}[section]
\newtheorem{lemma}[theorem]{Lemma}
\newtheorem{corollary}[theorem]{Corollary}
\newtheorem{innercustomthm}{Theorem}
\newenvironment{customthm}[1]
  {\renewcommand\theinnercustomthm{#1}\innercustomthm}
  {\endinnercustomthm}

\theoremstyle{definition}
\newtheorem{definition}[theorem]{Definition}

\newtheorem{example}[theorem]{Example}

\newtheorem{question}{Question}[subsection]
\newtheorem{conjecture}[theorem]{Conjecture}

\theoremstyle{remark}

\newcommand{\Z}{\mathbb{Z}}

\newcommand{\Ps}{\operatorname{PStab}}
\newcommand{\lk}{\operatorname{link}}

\title{Acylindrical Visual Splittings and The Tits Alternative For Artin Groups}
\date{}
\author{William D. Cohen}

\begin{document}
\begin{abstract}
    We give a necessary and sufficient condition on a visual splitting of an Artin group satisfying the conditions of two well known conjectures to be acylindrical, and demonstrate how this can be used to provide a large class of novel examples of Artin groups that satisfy the Tits alternative.
\end{abstract}
\maketitle
\section{Introduction}
Given a finite set $S$ and a graph $\Gamma_S$ whose vertices are labelled by $S$ whose every edge $\{u, v\}$ is labelled by an integer $m_{uv}\geq 2$, we define the \emph{Artin group} $A_S$ to be given by the presentation

 \[A_S:=\langle S\mid \{\underbrace{uvu\cdots uv}_{m_{uv}}= \underbrace{vuv\cdots vu}_{m_{uv}}  : \left\{u, v\right\}\in E(\Gamma)\}\rangle.\]
 
An Artin group is a natural generalisation of a braid group, with braid relations corresponding to Artin relations of length $3$. Artin groups may also be viewed as an variation on Coxeter groups, in that if we were to add to the above presentation the set of relations $\{s^2=1\mid s\in S\}$ we would obtain the Coxeter group on the same generating set and graph.

However, Artin groups tend to be significantly more mysterious than coxeter groups in general. For example, it is unknown in general when an Artin group has solvable word problem, when it has torsion or even when it has trivial centre. It is also open in general when an Artin group is acylindrically hyperbolic, a well studied and powerful generalisation of hyperbolicity that has been of much interest in recent years \cite{minasyanOsin13, Osin13,  Osin17}, and was the first such generalisation to encompass mapping class groups of hyperbolic surfaces (\cite{Bowditch08,MasurMinsky99}, see also \cite[Section~8]{Osin13}). 

The acylindrical hyperbolicity of Artin groups has been well studied, and is the subject of the following conjecture.

\begin{conjecture}
    Let $A_S$ be an Artin group, and $Z_S$ the centre of $A_S$. Then $A_S/Z_S$ is acylindrically hyperbolic.
\end{conjecture}

Some progress has been made towards this conjecture. For example, Vaskou proved that all 2-dimensional Artin groups with at least three generators are acylindrically hyperbolic \cite{Vaskou22}. Further, Charney, Martin and Morris-Wright have strongly linked this conjecture to the parabolic intersection conjecture (see Definition~\ref{def:PIP}), proving that if the latter conjecture holds, then every Artin group with a \emph{visual splitting} is acylindrically hyperbolic \cite{CMMR25}, where a visual splitting is a natural action on a tree that comes from the presentation graph $\Gamma_S$ of an Artin group $A_\Gamma$ and that exists if and only if this graph is not complete.

We focus in this paper on a stregthening of the acylindrical hyperbolicity question, namely concerning acylindrical actions on trees. The definition of an acylindrical action on a tree was first formulated by Sela \cite{Sela97}, and later generalised by Weidmann \cite{Weidmann07} to the following. 
\begin{definition}\label{Def:introkc} Let $G$ be a group acting by simplicial isometry on some simplicial tree $T$ and let $k\geq0$ and $C>0$ be integers. We say that the action of $G$ on $T$ is $(k, C)$\emph{-acylindrical} if the pointwise stabiliser of any edge path in $T$ of length at least $k$ contains at most $C$ elements.

If there exist such a $k$ and $C$ we simply say that the action of $G$ on $T$ is \emph{acylindrical}.
\end{definition}
We say that a group $G$ is \emph{acylindrically arboreal} if $G$ admits an acylindrical action on a tree that is \emph{non-elementary}, or equivalently that the action has no global fixed point or invariant axis. Noting that simplicial trees are all $0$-hyperbolic metric spaces, acylindrical arboreality may be viewed as a special case of acylindrical hyperbolicity, and so it is natural to ask the following question.

\begin{question}\label{q:main1}
    When is an Artin group $A_\Gamma$ acylindrically arboreal?
\end{question}

This question was answered for \emph{right-angled} Artin groups, or Artin groups defined by a presentation graph $\Gamma_S$ all of whose labels are $2$ \cite{Cohen23}. In this simpler case, the following holds as an immediate consequence of the proof of the cited theorem.

\begin{theorem}\cite[Theorem~1.3]{Cohen23}\label{cohen23}
    Let $A_S$ be a right-angled Artin group with presentation graph $\Gamma_S$. Then the following are equivalent:
    \begin{enumerate}
        \item $A_S$ admits a non-elementary acylindrical splitting;
        \item There is a visual splitting for $A_S$ that is non-elementary and acylindrical; and
        \item The graph theoretical diameter of $\Gamma_S$ is at least $3$.
    \end{enumerate}
\end{theorem}
It follows that for the right angled case, one need only consider visual splittings to decide acylindrical arboreality, and the existence of an acylindrical visual splitting reduces to a simply verified condition on the presentation graph. 

In the world of general Artin groups such a classification cannot hold. Indeed, with Example~\ref{badEx} we give an example of an acylindrically arboreal Artin group whose given presentation graph does not allow for any acylindrical visual splittings. However, a large part of Question~\ref{q:main1} can be decided by classifying the acylindricity of visual splittings of Artin groups. To this end, we are able to prove the following, which classifies the acylindricity of all visual splittings of Artin groups that satisfy certain conditions conjectured to hold for all Artin groups.
\begin{restatable}{theorem}{thmmain}\label{Thm}
Let $A_S$ be an Artin group with presentation graph $\Gamma_S$, and $X, Y\subseteq S$ such that $A_S=A_X*_{A_Z}A_Y$ is a non-trivial visual splitting. Assume further that there exist Artin groups $A_{X'}$ and $A_{Y'}$ with the parabolic intersection and ribbon properties (see Section~\ref{Artin} for relevant definitions) that contain $A_X$ and $A_Y$ respectively as special subgroups. Then $A_S=A_X*_{A_Z}A_Y$ is a non-elementary acylindrical splitting if and only if the neighbourhoods $N_{\Gamma_X}(X\setminus Z)$ and $N_{\Gamma_Y}(Y\setminus Z)$ are not connected by an odd labelled path in $\Gamma_S$.
\end{restatable}
\begin{restatable}{corollary}{cormain}\label{cor:main}
    Let $A_S$ be an Artin group with the parabolic intersection and ribbon properties and with presentation graph $\Gamma_S$. Then $A_S$ has an non-elementary acylindrical splitting arising as a visual splitting if and only if there exists two vertices $a$ and $b$ of $\Gamma_S$ whose neighbourhoods are not joined by a path with odd labels.
\end{restatable}
This condition on the neighbourhoods of vertices should be compared to the concept of \emph{separated vertices} given by the author in \cite[Definition~3.8]{Cohen23}. This corollary applies in particular to large-type Artin groups, or Artin groups whose presentation graphs have no edges labelled $3$, which are known to satisfy the parabolic intersection property by \cite[Theorem~1.3]{Blufstein} and the ribbon property by \cite[Corollary~4.12]{GODELLE200739}. We therefore present the following as a concrete consequence of our main theorem.
\begin{corollary}
    An Artin group $A_X$ with associated graph $\Gamma_X$ of large-type has a non-elementary acylindrical splitting arising as a visual splitting if and only if there exists two vertices $a$ and $b$ of $\Gamma_X$ whose neighbourhoods are not joined by path with odd labels.
\end{corollary}
Finally we have the following corollary, which should be viewed as a generalisation of the equivalence of (2) and (3) in Theorem~\ref{cohen23}.
\begin{corollary}
    Let $A_S$ be an Artin group with the parabolic intersection and ribbon properties and with presentation graph $\Gamma_S$. Assume further that $A_S$ is \emph{even}, so all labels in the presentation graph $\Gamma_S$ are even. Then $A_S$ has a non-elementary acylindrical splitting arising as a visual splitting if and only if the diameter of $\Gamma_S$ is at least $3$.
\end{corollary}
\subsection{Application to The Tits Alternative}
Part of the significance of acylindrical actions on trees is that they are acylindrical actions on hyperbolic spaces in which the elliptic subgroups are controlled --- every subgroup of a group acting acylindrically on a tree either acts with at least one loxodromic or fixes a point. This allows one to turn well known facts about acylindrically hyperbolic groups into powerful combination theorems. For example, a recent paper by Hagen, Martin and Sartori \cite{HMS25} proved that the Wise power alternative, a well-studied negative curvature property for groups, is inherited from the vertex stabilisers of an acylindrical action on a tree.

The main such property that we consider in this paper is the strong Tits alternative. We say that a group $G$ satisfies the \emph{Tits alternative} if for all finitely generated subgroups $H\leq G$, $H$ is either virtually soluble or contains a non-abelian free group. This property was introduced by Jacques Tits in 1972, who proved that all linear groups of any characteristic satisfy this condition \cite{TITS1972}. In the years since, many important groups have been shown to satisfy the Tits alternative, including hyperbolic groups (somewhat trivially, as the span of any two elements in a hyperbolic group is either virtually cyclic or contains a free group), mapping class groups of hyperbolic surfaces \cite{McCarthyTitsMapping}, automorphism groups of free groups \cite{BestvinaFeighnHandel00, BestvinaFeighnHandel05}, and all cocompactly cubulated groups \cite{sageev_wise_2005}. Similarly, we say that a group $G$ satisfies the \emph{strong Tits alternative} if for \emph{all} (not just finitely generated) subgroups $H\leq G$, $H$ is either virtually soluble or contains a non-abelian free group. This is a much stronger property, and was shown by Tits to be satisfied by linear groups in characteristic zero \cite{TITS1972}.

Much of the research into Artin groups in recent years has been into their non-positive curvature properties, and the strong Tits alternative is often viewed as such a property, so it is natural to ask which Artin groups satisfy this property (see \cite[Question~1]{Bestvina}, for example). This question has attracted great interest, and many partial results have been proven, but the general question is still very much open. We believe the following is a complete survey at the time of writing --- for definitions of the classes of Artin groups mentioned, see Section~\ref{Artin}.

\begin{itemize}
    \item Spherical Artin groups were shown to be linear of characteristic zero \cite[Theorem~1.1]{CohenWales2002}, so satisfy the strong Tits alternative by \cite[Theorem~1.1]{TITS1972} as above.
    \item Artin groups that are cocompactly cubulated will satisfy the tits alternative by \cite[Theorem~1.1]{sageev_wise_2005} as above, and indeed will satisfy the strong Tits alternative by the same result. An important class of cocompactly cubulated Artin groups is the class of right-angled Artin groups, but beyond this few Artin groups are known to cocompactly cubulate \cite[Theorems C and D]{Haettel},\cite[Theorem~1.1]{HuangJankiewiczPrzytycki16}, and conjecturally only Artin groups satisfying very strict conditions will enjoy this property \cite[Conjecture~A]{Haettel}.
    \item Artin groups of FC-type will satisfy the strong Tits alternative by \cite[Theorem~B]{MartinPrzytyckiFC}. This proof again uses a cocompact action on a CAT(0) cube complex, but instead of requiring that the action is proper Martin and Przytycki require that all stabilisers satisfy the strong Tits alternative, and placing a strong condition on the stabilisers of intersecting cubes.
    \item Artin groups acting on certain $2$-complexes were shown to satisfy the strong Tits alternative by \cite{OSAJDA2021107976}, including large type Artin groups \cite[Theorem~A.2]{OSAJDA2021107976} and some other $2$-dimensional examples.
    \item 2-dimensional Artin groups will satisfy the strong Tits alternative by \cite[Theorem~A]{MARTIN2024294}. We also mention \cite{MartinPrzytycki2dhyp}, where it was proved that hyperbolic-type 2-dimensional Artin groups will satisfy the Tits alternative by observing that such groups act acylindrically on a hyperbolic space such that the maximal elliptic subgroups can be classified.
\end{itemize}

Our strategy to expand this list will be to use the fact that in an acylindrical action on a tree maximal elliptic subgroups are easy to understand, as mentioned above. In particular, we believe the following is well known, although we include a proof in Section~\ref{sec:Tits} for completeness.
    \begin{lemma}
        Let $(\Gamma, \mathfrak{G})$ be an acylindrical graph of groups with fundamental group $G$. Then $G$ satisfies the strong Tits alternative if and only if the vertex group $G_v$ satisfies the strong Tits alternative for all $v\in V(\Gamma)$.
    \end{lemma}
This lemma, coupled with Theorem~\ref{Thm}, allows us to combine previously known examples of Artin groups satisfying the Tits alternative to acquire a wealth of new examples, and such an example is demonstrated with Example~\ref{ex:Tits}.
\subsection{Acknowledgements} This work was completed while the author was a PhD student at the University of Cambridge, supervised by Jack Button, and the author is very grateful to his supervisor for all of his help. The author would also like to thank Giovanni Sartori for several very helpful discussions, and Alexandre Martin and Mar\'ia Cumplido for their helpful comments. Finally, financial support from the Cambridge Trust Basil Howard Research Graduate Studentship is gratefully acknowledged. 
\section{Preliminaries}
\subsection{Groups Acting Acylindrically on Trees}
We recall some graph theoretical notation, which we will use throughout this paper.
\begin{definition}\label{def:graph}
    Let $\Gamma=(V(\Gamma), E(\Gamma))$ be a graph. We say that $\Gamma$ is \emph{finite} if $|V(\Gamma)|<\infty$, and we say that $\Gamma$ is \emph{simple} if $E(\Gamma)$ contains no loops or multiedges.

    For a vertex $v$ of a finite simple graph $\Gamma$ we define the \emph{link}, denoted $\lk_\Gamma(v)$, of $v$ to be the set of vertices $u\in V(\Gamma)\backslash\{v\}$ such that there exists an edge $e\in E(\Gamma)$ incident on both $u$ and $v$. For a subset $A$ of $V(\Gamma)$ we define $\lk_\Gamma(A)$ to be the intersection $\lk_\Gamma(A)=\bigcap_{v\in A}\lk_\Gamma(v)$. We define the \emph{neighbourhood} of a vertex $v\in V(\Gamma)$ to be $N_\Gamma(v)=\lk_\Gamma(v)\cup \{v\}$, and the neighbourhood of a set of vertices $A$ to be the union 
    $N_\Gamma(A)=\bigcup_{v\in A}N(v)$.
\end{definition} 

\begin{example} We will refer to the following standard collections of graphs.
\begin{enumerate}
        \item We say that $\Gamma=(V, E)$ is a \emph{complete} graph if $E$ contains every possible unordered pair of distinct elements in $V$, and we say that $\Gamma$ is \emph{discrete} if the edge set $E$ is empty. If $|V|=n$ we denote these graphs as $K_n$ and $O_n$ respectively.
        \item We define the \textit{$n$-path} $P_n$ for $n\geq 2$ to be the unique (up to isomorphism) connected graph on $n$ vertices with $n-1$ edges and maximum vertex degree two, and the \textit{$n$-cycle} $C_n$ for $n\geq 3$ to be the unique (up to isomorphism) connected graph on $n$ vertices with $n$ edges such that the degree of every vertex is two. 
    \end{enumerate}
\end{example}

We will assume the reader has some familiarity with Bass--Serre theory, and for a more detailed discussion we refer to \cite{stilwell2002trees, dicks2011groups}. Let $(\Gamma, \mathfrak{G})$ be a \emph{graph of groups}, where $\Gamma$ is a connected directed graph that may not be finite or simple and $\mathfrak{G}$ is the following data:
\begin{itemize}
    \item To every vertex $v\in V(\Gamma)$ we assign a \emph{vertex group} $G_v$, and to every edge $e\in E(\Gamma)$ we assign an \emph{edge group} $G_e$;
    \item To every edge $e\in E(\Gamma)$ we assign monomorphisms $d_0:G_e\rightarrow G_{i(e)}$ and $d_1:G_e\rightarrow G_{t(e)}$, where $i(e)$ and $t(e)$ are the initial and terminal vertices of $e$ in $\Gamma$ respectively.
\end{itemize}

We will use a slight abuse of notation to consider each vertex group $G_v$ as a subgroup of the fundamental group $\pi_1(\Gamma, \mathfrak{G})$ along the natural inclusion. Similarly, we will consider each edge group $G_e$ to be the subgroup of the fundamental group given by the image of $d_0(G_e)$ in the vertex group $G_{i(e)}$. We call a graph of groups \emph{trivial} if there exists some $v\in V(\Gamma)$ such that $G_v= \pi_1(\Gamma, \mathfrak{G})$, or \emph{non-trivial} otherwise. We say that a graph of groups $(\Gamma, \mathfrak{G})$ is a \emph{graph of groups decomposition}  or \emph{splitting} of a group $G$ if the fundamental group $\pi_1(\Gamma, \mathfrak{G})$ is isomorphic to $G$. We denote by $T(\Gamma, \mathfrak{G})$ the \emph{Bass--Serre tree} associated to the splitting, on which $G$ acts naturally by isometry with respect to the edge metric and without inversion \cite[Section~I.5.3]{stilwell2002trees}.

The assumption that any action on a tree is simplicial and without inversion is easy to guarantee, so we will assume from now on that all actions on trees are by simplicial isometry and without inversion.

As in \cite[Section~I.5.4]{stilwell2002trees}, an action on a tree will give rise to a \emph{quotient graph of groups decomposition $(T/G, \mathfrak{G}$)} of $G$, where the vertex or edge group of a vertex or edge of $T/G$ is defined to have the isomorphism type of the stabiliser of any preimage of that vertex or edge in $T$, and the edge monomorphisms are defined similarly.

%
%
We now formally define acylindrical arboreality
\begin{definition}\cite[Introduction]{Weidmann07}\label{Def:kc} Let $G$ be a group acting on some tree $T$ and let $k\geq0$ and $C>0$ be integers. We say that the action of $G$ on $T$ is $(k, C)$\emph{-acylindrical} if the pointwise stabiliser of any edge path in $T$ of length at least $k$ contains at most $C$ elements.

We say that the action of $G$ on $T$ is \emph{acylindrical} if there exist constants $k$ and $C$ such that the action is $(k, C)$-acylindrical.
\end{definition}

This definition of acylindricity will agree with the more coarse-geometric definition due to Bowditch \cite[Introduction]{Bowditch08} when the latter definition is restricted to actions on trees, in a result essentially due to Osin and Minasyan \cite[Lemma~4.2]{minasyanOsin13}, and an explicit proof can be found in \cite[Theorem~2.17]{Cohen23}.

\begin{definition} \label{def:AA} We say that a group $G$ is \emph{acylindrically arboreal} if $G$ acts acylindrically on some tree $T$ with no global fixed points or invariant lines. This action will give rise to a quotient graph of groups decomposition of $G$, which we will call a \emph{non-elementary acylindrical} splitting of $G$. 
\end{definition}

Similarly, this definition of a non-elementary acylindrical action will agree with that of Bowditch when the latter is restricted to trees. It follows that any acylindrically arboreal group is acylindrically hyperbolic.
\subsection{Artin Groups}\label{Artin}
\begin{definition} Let $\Gamma_S$ be a \emph{labelled} finite simple graph with vertex set $S$, or equivalently a graph where each edge $\{u, v\}$ is labelled by an integer $m_{uv}\geq 2$. Then the \emph{Artin group} over $\Gamma$ is the group
    \[A_S:=\langle S\mid \{\underbrace{uvu\cdots uv}_{m_{uv}}= \underbrace{vuv\cdots vu}_{m_{uv}}  : \left\{u, v\right\}\in E(\Gamma)\}\rangle.\]

    We call the graph $\Gamma_S$ the \emph{presentation graph} of $A_S$.
\end{definition}

We will sometimes be interested in the graph representing an Artin group $G$ that ignores edges whose label is $2$ and includes edges that are not included in $\Gamma_S$ with the label $\infty$. Rigorously, we define the graph $\overline{\Gamma}_S$ to be the labelled graph whose vertex set is $V(\Gamma)$ and whose edge set is the union of \[E_{>2}(\Gamma_S) = \{\{u, v\}\in\Gamma\mid m_{u, v}>2\},\] with labels the same as in $\Gamma_S$, and \[\overline{E}(\Gamma_S):=\{\{u, v\}\mid u, v\in V(\Gamma_S), \{u, v\}\notin E(\Gamma_S)\}\] all of which are labelled $\infty$.

The graph $\overline{\Gamma}_S$ is called the \emph{Dynkin diagram} associated to $A_S$. Finally, the following are important classes of subgroups of Artin groups.

\begin{definition}
    Let $A_S$ be an Artin group and $X\subseteq S$ be a subset of $S$. We say that the subgroup of $A_S$ generated by $S$ is the \emph{special subgroup on $X$}, denoted $A_X$. If a subgroup $H$ of $A_S$ is conjugate to some special subgroup, we say that $H$ is a \emph{parabolic subgroup} of $A_S$.
\end{definition}

\begin{example}
    The following are standard families of Artin groups to which we will refer in this paper.
    \begin{enumerate}
        \item An Artin group $A_S$ is called \emph{dihedral} if $S$ contains exactly two elements and $\Gamma_S$ contains only a single edge.
        \item On the same graph $\Gamma_S$ we may define the \emph{Coxeter group} $C_S$ to be the group with the same presentation but for the added condition that each element of $S$ has order $2$. We say that an Artin group $A_S$ is \emph{spherical} if the corresponding Coxeter group $C_S$ is finite. For example, all dihedral Artin groups are spherical as their corresponding Coxeter groups are finite dihedral groups. Note that the Coxeter group on the discrete graph with two vertices is the infinite dihedral group, so an Artin group $A_S$ can be spherical only if $\Gamma_S$ is a complete graph, although this is not an equivalence.
        \item More generally, we say that an Artin group $A_S$ is of \emph{finite-clique-type} (or \emph{FC-type}) if every complete subgraph of $\Gamma_S$ represents a spherical special subgroup of $A_S$.
        \item In the other direction, we say that an Artin group $A_S$ is $2$-dimensional if every spherical special subgroup is generated by a subset of $S$ of size at most $2$. A particularly well studied subclass of $2$-dimensional Artin groups is the class of \emph{large-type} Artin groups, where the label of every edge of $\Gamma_S$ is at least $3$, or equivalently if $\overline{\Gamma}_S$ is complete.
        \item Finally, we say that an Artin group $A_S$ is \emph{$(2, 2)$-free} if every vertex $v$ on $V(\Gamma)$ has at most one edge incident on it with the label $2$.
        \end{enumerate}
\end{example}
\begin{definition}
    We say that an Artin group $A_S$ with presentation graph $\Gamma_S$ is \emph{reducible} if its Dynkin diagram $\overline{\Gamma}_S$ is disconnected, and we say that $A_S$ is irreducible otherwise. The \emph{irreducible components} of $S$ are the maximal subsets of $S$ that correspond to irreducible Artin groups.
\end{definition}
It will often be necessary to consider the spherical and non-spherical irreducible components of an Artin group separately. We therefore fix the following notation.
\begin{definition}
    Let $A_S$ be an Artin groups with presentation graph $\Gamma_S$, and let $X\subseteq S$ be a subset of the vertices of $\Gamma_S$. We denote by $X_s$ and $X_{as}$ respectively the union of the spherical irreducible components of $X$ and the union of the non-spherical components of $X$. 
\end{definition}
Finally, we will use the following definition of a the vertices that are in some sense orthogonal to a given set.
\begin{definition}
    Let $A_S$ be an Artin group with presentation graph $\Gamma_S$, and let $X\subseteq S$ be a subset of the vertices of $\Gamma_S$. We define $X^\perp$ to be given by 
    \[X^\perp = \{s\in S: \forall t\in X, \{s, t\}\in E(\Gamma_S)\text{ and }m_{st}=2\}.\]
    In particular, note that if $X=\emptyset$ then $X^{\perp}=S$.
    
\end{definition}
We have the following theorem due to Paris, which shows exactly when individual generators of an Artin group are conjugate.
\begin{theorem}\label{thm:oddpathconj}\cite[Theorem~4.2]{PARIS1997369} Let $A_S$ be an Artin group on the graph $\Gamma_S$. Then two generators $a, b\in A_S$ are conjugate in $A_S$ if and only if there is a path between $a$ and $b$ in $\Gamma_S$ with odd labels.
\end{theorem}

Special subgroups and their intersections are well understood. In particular, we have the following theorem of Van der Lek.

\begin{theorem}\cite[Theorem~4.13]{vanderLek}\label{Van der Lek}
    Let $A_S$ be an Artin group, and $X\subseteq S$. If $\Gamma_X$ is the subgraph of $\Gamma_S$ induced by $X$, then the special subgroup $A_X$ is isomorphic to the Artin group of $\Gamma_X$. Furthermore, if $Y\subseteq S$ then $A_X\cap A_Y=A_{X\cap Y}$.
\end{theorem}

In contrast, parabolic subgroups and their intersections have been a subject of much interest in the study of Artin groups. We have the following important property of parabolic subgroups of an Artin group.
\begin{definition}\label{def:PIP}
    We say that an Artin group $A_S$ has the \emph{parabolic intersection property (PIP)} if the intersection of any two parabolic subgroups of $G$ is a parabolic subgroup of $G$.
\end{definition} 

It is conjectured that every Artin group will have the parabolic intersection property, and it is known for many groups, for example for even Artin groups of FC-type \cite[Theorem~1.1]{Antolín_Foniqi_2022} and $2$-dimensional $(2, 2)$-free Artin groups \cite[Theorem~1.3]{Blufstein}. In particular, this latter class includes large-type Artin groups.

In an Artin group $A_X$ with the parabolic intersection property, we can put strong conditions on such intersections. Indeed, we have the following theorem of Blufstein and Paris.
\begin{theorem}\label{standard}\cite[Theorem~1.1]{blufsteinParis}
    Let $A_S$ be an Artin group, and let $X, Y\subseteq S$. If there exists $g\in A_S$ such that $gA_Xg^{-1}\leq A_Y$ then there exists $Z\subseteq Y$ and $h\in A_Y$ such that $gA_Xg^{-1}=hA_Zh^{-1}$. 

    In particular, if $A_S$ has the property PIP then for all $X, Y\subseteq S$, $g\in A_s$ there exists $Z\subseteq Y$ and $h\in A_Y$ such that $gA_Xg^{-1}\cap A_Y=hA_Zh^{-1}$.
\end{theorem}
To define the last properties that we will require in this paper, we recall the following definitions due to Godelle \cite[Section~1]{GODELLE200739}.

For an Artin group $A_S$, we define the category $\operatorname{Conj}(S)$ as follows. We set the objects of $\operatorname{Conj}(S)$ to be all subsets of $S$, and set the morphisms between $X$ and $Y\subseteq V(S)$ to be in bijection with the elements of $G$ such that $gXg^{-1}=Y$ (note here that we are interested in conjugating the sets of generators to each other rather than simply the subgroups they generate). We denote the set of morphisms $X$ to $Y$ in $\operatorname{Conj}(S)$ by $\operatorname{Conj}(S; X, Y)$.

Consider now the monoid $A_S^+$ of positive words in $A_S$, which as a monoid has the same presentation as $A_S$. We may partially order the elements of $A_S^+$ by left division, so for $a, b\in A_S^+$ we write that $a\preceq b$ if there exists $c\in A_S^+$ such that $ac=b$.

\begin{lemma}\cite[Theorem~5.6]{BrieskornSaito}
    Let $A_S$ be an Artin group. Then the set $S$ has a least common multiple with respect to $\preceq$ in $A_S^+$ if and only if $A_S$ is spherical. In such a case, we denote this least common multiple by $\Delta_S$.
\end{lemma}

\begin{example}
    If $A_S$ is a dihedral Artin group with two generators $u$ and $v$ and a single relation of length $m\in\mathbb{N}$, then $A_S$ will be spherical as the corresponding Coxeter group will be the finite dihedral group $D_{2m}$. In this case, $\Delta_S=\underbrace{uvu\cdots uv}_{m}=\underbrace{vuv\cdots vu}_{m}$
\end{example}

Now let $A_S$ be an Artin group and $X, Y\subseteq S$. We say that an element $g$ of $A_S$ is an \emph{elementary $(X, Y)$-ribbon} if $gXg^{-1}=Y$ and there exists $t\in S\setminus X$ such that:
\begin{enumerate}[(R1)]
\item The vertex set of the connected component $U$ of $\overline{\Gamma}_S(X\cup\{t\})$ containing $t$ generates a spherical subgroup of $A_S$; and
\item We have that $g=\Delta_{U}^{-1}\Delta_{U\setminus\{t\}}$.
\end{enumerate}

We define the category $\operatorname{Ribb}(S)$ to be the smallest subcategory of $\operatorname{Conj}(S)$ containing the same set of objects and all morphisms that correspond to elementary ribbons. The set of morphisms in $\operatorname{Ribb}(S)$ therefore correspond to finite compositions of morphisms coresponding to elementary ribbons. We denote the set of morphisms $X$ to $Y$ in $\operatorname{Ribb}(S)$ by $\operatorname{Ribb}(S; X, Y)$, and call an element of $A_S$ an $(X, Y)$-ribbon if it corresponds to an element of $\operatorname{Ribb}(S; X, Y)$.

We will refer to the set of elements of $A_S$ that correspond to elements of $\operatorname{Conj}(S; X, Y)$ and $\operatorname{Ribb}(S; X, Y)$ simply as $\operatorname{Conj}(S; X, Y)$ and $\operatorname{Ribb}(S; X, Y)$ respectively.

\begin{definition} \cite[Definition~4.1]{GODELLE200739}\label{def:ribbon}
    We say that an Artin group $A_S$ has the \emph{ribbon property (RP)} if for all $X, Y\subseteq S$ and $g\in A_S$ we have that $g A_Xg^{-1}\subseteq A_Y$ if and only if $X_{as}\subset Y$ and $g\in A_Y\cdot \operatorname{Ribb}(X^{\perp}_{as}; X, R)$ for some $R\subset Y$.
\end{definition}

%
%
%
It is again conjectured that all Artin group have property RP \cite[Conjecture~4.2]{GODELLE200739}, and it is known for many Artin groups, for example for $2$-dimensional Artin groups \cite[Corollary~4.12]{GODELLE200739} and Artin groups of FC-type \cite[Theorem~3.2]{GodelleFC}. In particular, every large-type Artin group is (2, 2)-free and $2$-dimensional, so we have the following lemma.
\begin{lemma}\cite{Blufstein, GODELLE200739}
    Every large-type Artin group has properties PIP and RP.
\end{lemma}

Finally, we define the most important construction in this paper, the visual splitting. 

\begin{definition}
    Let $A_S$ be an Artin group with presentation graph $\Gamma_S$. An amalgam decomposition on $A_S$ is a \emph{visual splitting} if there exist $X, Y\subset S$ such that $X\cup Y=S$ and such that our amalgam decomposition is of the form $A_S=A_X*_{A_{X\cap Y}}A_Y$.

    For a given $X$ and $Y$ such a splitting exists if and only if $X\cap Y$ separates $\Gamma_S$.
\end{definition}

\section{Proofs of Main Theorem and Corollaries}
As the strong conditions on the Artin groups in question are only required in one direction for Theorem~\ref{Thm} we prove each direction separately with the following lemmas.

\begin{lemma}\label{lem:onlyif}
    Let $A_X*_{A_Z}A_Y$ be a non-trivial visual splitting of an Artin group $A_S$ with presentation graph $\Gamma_S$, and assume that there exists a path in $\Gamma_S$ between $N_\Gamma(X\setminus Z)$ and $N_\Gamma(Y\setminus Z)$. Then the splitting $A_X*_{A_Z}A_Y$ is not acylindrical.
\end{lemma}
\begin{proof}
    Let $x\in X\setminus Z$ and $y\in Y\setminus Z$ such that there exist neighbours $x'$ of $x$ and $y'$ of $y$ with a path $P=(p_0=x',p_1,...,p_n=y')$ in $\Gamma_S$ from $x$ to $y$ with odd labels. Using that $Z$ separates $\Gamma_S$ by definition of a visual splitting, may assume possibly by passing to a subpath that $p_i\in Z$ for all $0\leq i\leq n$.

    We claim that the centraliser of $x'$ contains elements of both $A_X\setminus A_Z$ and $A_Y\setminus A_Z$. Indeed, $A_{\{x, x'\}}$ is a dihedral Artin group, so is either has a non-trivial centre generated by a single element \cite[Theorem~7.2]{BrieskornSaito} which we call $z_{x, x'}$, or is isomrprphic to $\Z^2$, in which case we choose $z_{x, x'}=x$. In both cases, $z_{x, x'}$ lies in $A_X\setminus A_Z$ and centralises $x'$, and there similarly exists $z_{y, y'}\in A_Y\setminus A_Z$ that centralises $y'$.

    By choice of $y'$ and $x'$ they are connected by an odd path in $\Gamma_S$, so by Theorem~\ref{thm:oddpathconj} they are conjugate in $A_Z$, so there exists some element $g\in A_Z$ such that $gx'g^{-1}=y'$. Thus $g^{-1}z_{y, y'}g$ is an element of $A_y\setminus A_Z$ that centralises $x'$, and the claim is proven. 
    
    The element $g^{-1}z_{y, y'}gz_{x, x'}$ then acts loxodromically on the Bass-Serre tree $T$ associated to the visual splitting $A_X*_{A_Z}A_Y$ of $A_S$, but centralises $x'$, an elliptic element, implying by $\langle x'\rangle$ fixes an unbounded set in $T$. However, Theorem~\ref{Van der Lek} tells us that $\langle x'\rangle\cong\Z$ which is infinite, and so it follows that the action of $A_S$ on $T$ cannot be acylindrical and the result follows.
\end{proof}

The following lemma will be used to control the stabilisers of paths of length two in visual splittings of well-behaved Artin groups.
\begin{lemma}\label{lem:Stabs}
    Let $A_S$ be an Artin group with presentation graph $\Gamma_S$ and properties PIP and RP, and let $X\subseteq S$ and let $Z\subseteq X$. Let $g\in A_X\setminus A_Z$. Then there exists $Z_1\subseteq Z$ such that $A_Z\cap gA_Zg^{-1}$ is conjugate in $A_Z$ to the special subgroup $A_{Z_1}$ and such that all elements $z\in Z_1$ are connected to some element of $N_{\Gamma_X}(X\setminus Z)$ by a path with odd labels. 
    
    Moreover, if $Z'\subseteq X$ contains some vertex not connected to $N_{\Gamma_X}(X\backslash Z)$ by a path with odd labels then $A_{Z'}$ is not conjugate in $A_S$ into $A_Z\cap gA_Zg^{-1}$.
\end{lemma}
\begin{proof}
    The subgroup $A_Z\cap gA_Zg^{-1}$ is the intersection of two parabolic subgroups of $A_S$, so is itself a parabolic subgroup of $A_S$ by PIP. Furthermore, $A_Z\cap gA_Zg^{-1}\leq A_Z$, so by Theorem~\ref{standard} we have that there exists $Z_1\subseteq Z$ and $h_1\in A_Z$ such that $A_Z\cap gA_Zg^{-1}= h_1A_{Z_1}h_1^{-1}$. Similarly, $g^{-1}A_Zg\cap A_Z\leq A_Z$, so there exists $Z_2\subseteq Z$ and $h_2\in A_Z$ such that $gA_Zg^{-1}\cap A_Z= h_2A_{Z_2}h_2^{-1}$. Therefore we have that 
    \begin{align*}h_1A_{Z_1}h_1^{-1}&=A_Z\cap gA_Zg^{-1}\\&=g\left(g^{-1}A_Zg\cap A_Z\right)g^{-1}\\&=g^{-1}h_2A_{Z_2}h_2^{-1}g,\end{align*}
    and so $A_{Z_2}=h_2^{-1}gh_1A_{Z_1}h_1^{-1}g^{-1}h_2$.

    Now let $Y = \left(S\setminus X\right)\cup Z$. The group $A_S$ has RP, so by Definition~\ref{def:ribbon} there exists $R\subset Z_2$ and $h_3\in A_{Z_2}$ such that \[h_3^{-1}h_2^{-1}gh_1\in \operatorname{Ribb}((Z_1)^\perp_{as}; Z_1, R)\cap A_X\subseteq \operatorname{Ribb}(S; Z_1, R)\cap A_X. \] Assume for contradiction that there exists some element $z\in Z_1$ that is not connected to any element of $N_{\Gamma_X}(X\setminus Z)$ by a path in $Z$ with odd labels. We will show by induction that all ribbons in $A_S$ originating from $Z_1$ are contained in $A_Y$, and so $h_2^{-1}gh_1h_3^{-1}\in A_Y\cap A_X=A_Z$.

    For the base case, let $T_1\subseteq Z$ and let $r\in \operatorname{Ribb}(S;, Z_1, T_1)$ be an elementary ribbon conjugating $Z_1$ to $T_1$. Then there exists $t\in S$ such that the connected component $U$ of $\overline{\Gamma}_{Z_1\cup\{t\}}$ corresponds to a spherical subgroup of $A_S$ and $r=\Delta_U^{-1}\Delta_{U\setminus \{t\}}$. In particular, $t$ must be a neighbour to each element of $Z_1$ in $\Gamma_S$. It follows that $t\notin X\setminus Z$ as $z$ is not in the neighbourhood of $X\setminus Z$ by assumption, and so $r\in A_Y$ as required. Now consider the image $z_1$ of $z$ under conjugation by $r$, which will be an element of $Y$ by the fact that $r, z\in A_Y$. By Theorem~\ref{thm:oddpathconj}, there must be a path in $\Gamma_S$ with odd labels between $z$ and $z_1$. It follows that there is no odd labelled path in $\Gamma_S$ connected $z_1$ to the neighbourhood of $N_{\Gamma_S}(X\setminus Z)$.

    Now assume for induction that every product $r$ of $i$ elementary ribbons originating from $Z_1$ is an element of $A_Y$, and that the conjugate of $Z_1$ by $r$ contains an element $z_i\in Z$ not connected to the neighbourhood of $N_{\Gamma_S}(X\setminus Z)$ by an odd path in $\Gamma_Z$. Then, as above, any elementary ribbon originating from $T_i=rZ_1r^{-1}$ will be contained in $A_Y$ and the image of $T_i$ under conjugation by an elementary ribbon will contain an element $z_{i+1}$ not connected to $N_{\Gamma_X}(X\setminus Z)$ by an odd path. It therefore follows by induction that any ribbon originating from $Z_1$ is contained within $A_Y$.

    We therefore have that $h_2^{-1}gh_1h_3^{-1}\in A_Y\cap A_X=A_Z$, but by construction $h_1, h_2$ and $h_3$ are in $A_Z$, and so $g$ must also be in $A_Z$. This contradicts the assumption that $g\notin A_Z$, and so it follows that every element of $Z_1$ is connected to the neighbourhood of $X\setminus Z$ by an odd path as claimed.

    Finally, assume that for some $Z'\subseteq Z$ the special subgroup $A_{Z'}$ is conjugate in $A_S$ into $A_{Z_1}$. Then $Z'$ is conjugate by ribbons in $A_S$ to some subset $Z'_1\subseteq Z_1$ by assumption that $A_S$ has property RP, and so by the same argument each vertex of $Z'$ must be connected to the neighbourhood of $X\setminus Z$ by an odd path in $\Gamma_Z$ as required.
\end{proof}

We are now ready to prove Theorem~\ref{Thm}, which we reformulate in the language developed above.
\begin{customthm}{\ref{Thm}}
Let $A_S$ be an Artin group with presentation graph $\Gamma_S$, and $X, Y\subseteq S$ such that $A_S=A_X*_{A_Z}A_Y$ is a non-trivial visual splitting. Assume further that there exist Artin groups $A_{X'}$ and $A_{Y'}$ with properties PIP and RP that contain $A_X$ and $A_Y$ respectively as special subgroups. Then $A_S=A_X*_{A_Z}A_Y$ is non-elementary acylindrical splitting if and only if the neighbourhoods $N_{\Gamma_X}(X\setminus Z)$ and $N_{\Gamma_Y}(Y\setminus Z)$ are not connected by an odd path in $\Gamma_S$.
\end{customthm}
\begin{proof}
    For the only if direction, let $A_X*_{A_Z}A_Y$ be a non-trivial visual splitting of $A_S$ such that $N_{\Gamma_X}(X\setminus Z)$ and $N_{\Gamma_Y}(Y\setminus Z)$ are connected by a path in $\Gamma_S$ with odd labels. It then follows immediately from Lemma~\ref{lem:onlyif} that this splitting is not acylindrical. We therefore proceed with the if direction. Let $A_S$ be the Artin group with presentation graph $\Gamma_S$, and let $A_X*_{A_Z}A_Y$ be a non-trivial visual splitting of $A_S$ such that $N_{\Gamma_X}(X\setminus Z)$ and $N_{\Gamma_Y}(Y\setminus Z)$ are not connected by a path in $\Gamma_S$ with odd labels. We will show that this splitting is $(3, 1)$-acylindrical.

    \begin{figure} 
        \centering
        \begin{tikzpicture}[->,shorten >=1pt,auto,node distance=3cm,
                thick]

            \node (1) at (-1, 0) {$A_X$};
            \node (2) at (1, 0) {$A_Y$};
            \node (3) at (-2.5, -1) {$g_1A_Y$};
            \node (4) at (2.5, 1) {$g_2A_X$};

            \path[-]
            (1) edge node {$A_Z$} (2)
                edge node {$g_1A_Z$} (3)
            (2) edge node {$g_2A_Z$} (4);
        \end{tikzpicture}
        \caption{\label{linefig}A generic 3-path in the Bass--Serre tree of the visual splitting $A_X*_{A_Z}\nolinebreak A_Y$ can be assumed to use $A_Z$ as its middle edge as the action is edge-transitive and by isometries.}
    \end{figure}
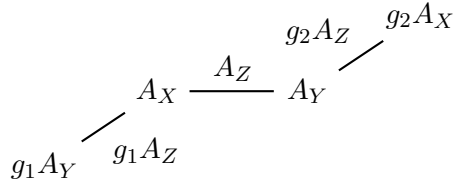 
    
    Consider a $3$-path $P$ in the Bass-Serre tree $T$ associated to this splitting. By edge transitivity of the action of $A_S$ on $T$ we may assume that the central edge of this path is labelled $A_Z$ as shown in Figure~\ref{linefig}, and that there exist $g_1\in A_X\setminus A_Z$, $g_2\in  A_Y\setminus A_Z$ such that the remaining two edges are labelled $g_1A_Z$ and $g_2A_Z$. The stabiliser of $P$ is therefore given by
    \begin{align*}
    \Ps_{A_S} (P)&=g_1A_Zg_1^{-1}\cap A_Z \cap g_2A_Zg_2^{-1}\\&=\left(g_1A_Zg_1^{-1}\cap A_Z\right)\cap\left(A_Z \cap g_2A_Zg_2^{-1}\right). 
    \end{align*}
    
    By Lemma~\ref{lem:Stabs} there exist $Z_1, Z_2\subseteq Z$ such that $g_1A_Zg_1^{-1}\cap A_Z$ is conjugate in $A_Z$ to $A_{Z_1}$ and $g_2A_Zg_2^{-1}\cap A_Z$ is conjugate in $A_Z$ to $A_{Z_2}$, and such that each vertex in $Z_1$ is connected by a path with odd labels in $Z$ to an element of $N_{\Gamma_X}(X\setminus Z)$ and each vertex in $Z_2$ is connected by a path with odd labels in $Z$ to an element of $N_{\Gamma_Y}(Y\setminus Z)$. There is no odd labelled path from $N_{\Gamma_X}(X\setminus Z)$ to $N_{\Gamma_Y}(Y\setminus Z)$ by assumption on $X$ and $Y$, so there can be no odd labelled path from any vertex of $Z_1$ to $N_{\Gamma_Y}(Y\setminus Z)$ or from any vertex of $Z_2$ to $N_{\Gamma_X}(X\setminus Z)$, so by the second part of Lemma~\ref{lem:Stabs} and Theorem~\ref{standard} their intersection must be trivial, and so this action $(3, 1)$-acylindrical as required. Finally, we observe that an Artin group on a presentation graph with at least two vertices can never by virtually cyclic, and so the fact that the given splitting is non-elementary acylindrical follows from \cite[Theorem~1.1]{Osin13}.
\end{proof}
Similarly, we prove Corollary~\ref{cor:main}, which we restate here for clarity.
\cormain*
\begin{proof}
First assume there exists two non-adjacent vertices $a$ and $b$ of $\Gamma_S$ whose neighbourhoods are not joined by path with odd labels. Then the visual splitting $A_S\cong A_{S\setminus\{a\}}*_{A_{S\setminus\{a, b\}}}A_{S\setminus\{b\}}$ satisfies the conditions of Theorem~\ref{Thm} with $A_{X'}=A_{Y'}=A_S$.

Now assume that there exists a non-trivial visual splitting $A_S\cong A_X*_{A_Z}A_Y$ of $A_S$ which is acylindrical. By non-triviality of the given splitting there exists $a\in X\setminus Z$ and $b\in Y\setminus Z$, and by Theorem~\ref{Thm} again with $A_{X'}=A_{Y'}=A_S$ there must be no odd path in $\Gamma_S$ between the neighbourhoods $N_{\Gamma_S}(X\setminus Z)$ and $N_{\Gamma_S}(Y\setminus Z)$ by acylindricity. It follows that the links of $a$ and $b$ in $\Gamma_S$ are not joined by a path in $\Gamma_S$ with odd labels as required.
\end{proof}

We finish this section with an example that demonstrates that, unfortunately, visual splittings do not paint a complete picture of the acylindrical arboreality of Artin groups.
\begin{example}\label{badEx}
    Let $A_S$ be the Artin group whose presentation graph $\Gamma_S$ is a copy of $P_3$ where one edge is labelled $2$ and the other is labelled $3$. Then $A_S$ has exactly one non-trivial visual splitting whose edge groups are both dihedral Artin groups which are spherical and thus satisfy the parabolic intersection property and the ribbon property. We may therefore apply Theorem~\ref{Thm} to see that this splitting is not acylindrical, and so $A_S$ has no acylindrical visual splitting. However, the tree of cylinders associated to this splitting by \cite{GuiradelLevitt11} is non-trivial, and will be acylindrical by \cite[Proposition~4.5]{JonesMangiorgioSartori}, for example. The group $A_S$ contains a copy of $\Z^2$ arising as the inclusion of the special subgroup on the edge labelled $2$, so in particular is not virtually cyclic and so by \cite[Theorem~1.1]{Osin13} the action of $A_S$ on is non-elementary and so $A_S$ is acylindrically arboreal.
\end{example}

\section{The Tits Alternative}
\label{sec:Tits}
    The main application we present here of our classification of visual splittings is to the Tits alternative for Artin groups. We recall the following definition.
    \begin{definition}
        Let $G$ be a group. We say that $G$ satisfies the \emph{strong Tits alternative} if for all subgroups $H\leq G$, $H$ is either virtually soluble or contains a non-abelian free group.
    \end{definition}
    The Tits alternative is a non-positive curvature property of sorts, and we may study it using acylindrical actions on hyperbolic spaces by classifying maximal elliptic subgroups (see \cite{MartinPrzytycki2dhyp}, for example). In an acylindrical action on a tree, these subgroups are simply vertex stabilisers, and as such we have the following well known result, of which we include a proof for completeness.
    \begin{lemma}
        Let $(\Gamma, \mathfrak{G})$ be an acylindrical graph of groups with fundamental group $G$. Then $G$ satisfies the strong Tits alternative if and only if the vertex group $G_v$ satisfies the strong Tits alternative for all $v\in V(\Gamma)$.
    \end{lemma}
    \begin{proof}
        First assume that for all $v\in V(\Gamma)$, $G_v$ satisfies the strong Tits alternative. Let $H\leq G$, and consider the action of $H$ on $T=T(\Gamma, \mathfrak{G})$, the Bass-Serre tree of the acylindrical splitting $(\Gamma, \mathfrak{G})$. By \cite[Theorem~1.1]{Osin13} we have three cases to consider.
        \begin{enumerate}
            \item The action of $H$ on $T$ is elliptic. Then there exists some vertex $v'\in T$ that $H$ stabilises, so there exists $v\in V(\Gamma)$ such that $H$ is conjugate in $G$ into $G_v$. Thus $H$ is isomorphic to a subgroup of a group satisfying the strong Tits alternative by assumption, and so $H$ is either virtually soluble or contains a non-abelian free subgroup as required.
            \item The action of $H$ on $T$ is \emph{lineal}, or fixes some line in $T$ setwise. In this case $H$ is virtually cyclic by \cite[Theorem~1.1]{Osin13}, and so $H$ is virtually soluble.
            \item The action of $H$ on $T$ is non-elementary, and $H$ is acylindrically arboreal. Then $H$ contains a non-elementary free subgroup \cite[Theorem~6.14]{DahmaniGuirardelOsin17}.
        \end{enumerate}
       Thus $G$ satisfies the strong Tits alternative. For the other direction, the strong Tits alternative is inherited by subgroups by definition, and so if $G$ satisfies the strong Tits alternative then so must every vertex group $G_v$ for $v\in V(\Gamma)$ as required.
    \end{proof}
    Applied to visual splittings of Artin groups this has the following immediate consequence.
    \begin{corollary}\label{cor:titsArtin}
        Let $A_S$ be an Artin group with an acylindrical visual splitting $A_X*_{A_Z}A_Y$. Then $A_S$ satisfies the strong Tits alternative if and only if $A_X$ and $A_Y$ satisfy the strong Tits alternative.
    \end{corollary}

    Using Theorem~\ref{Thm}, this allows us to construct many and varied examples of Artin groups with the strong Tits alternative by combining previously known examples.
    \begin{example}\label{ex:Tits}
        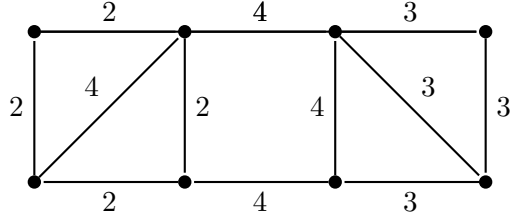
\begin{figure}
            \centering

            \begin{tikzpicture}[->,shorten >=1pt,auto,node distance=3cm,thick]

                \node[draw, circle, fill, inner sep=1.5pt] (a) at (0, 0) {};
                \node[draw, circle, fill, inner sep=1.5pt] (b) at (2, 0) {};
                \node[draw, circle, fill, inner sep=1.5pt] (c) at (4, 0) {};
                \node[draw, circle, fill, inner sep=1.5pt] (d) at (6, 0) {};
                \node[draw, circle, fill, inner sep=1.5pt] (e) at (0, 2) {};
                \node[draw, circle, fill, inner sep=1.5pt] (f) at (2, 2) {};
                \node[draw, circle, fill, inner sep=1.5pt] (g) at (4, 2) {};
                \node[draw, circle, fill, inner sep=1.5pt] (h) at (6, 2) {};
            
                \path[-]
                    (a) edge node {$2$} (e)
                        edge node {$4$} (f)
                    (b) edge node {$2$} (a)
                    (c) edge node {$4$} (b)
                        edge node {$4$} (g)
                    (d) edge node {$3$} (c)
                    (e) edge node {$2$} (f)
                        edge node {$4$} (h)
                    (f) edge node {$4$} (g)
                        edge node {$2$} (b)
                    (g) edge node {$3$} (h)
                        edge node {$3$} (d)
                    (h) edge node {$3$} (d);
          
            \end{tikzpicture}
            \caption{\label{weirdgph}The presentation graph $\Gamma_S$ of an Artin group that is neither $2$-dimensional, spherical or FC-type.}
        \end{figure}
        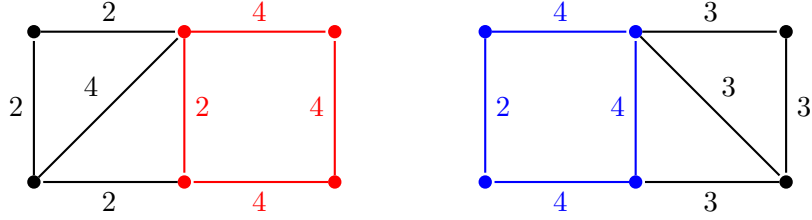
\begin{figure}
            \centering

            \begin{tikzpicture}[->,shorten >=1pt,auto,node distance=3cm,thick]

                \node[draw, circle, fill, inner sep=1.5pt] (a) at (0, 0) {};
                \node[draw, circle, fill, inner sep=1.5pt, red] (b) at (2, 0) {};
                \node[draw, circle, fill, inner sep=1.5pt, red] (c) at (4, 0) {};
                \node[draw, circle, fill, inner sep=1.5pt] (e) at (0, 2) {};
                \node[draw, circle, fill, inner sep=1.5pt, red] (f) at (2, 2) {};
                \node[draw, circle, fill, inner sep=1.5pt, red] (g) at (4, 2) {};

                \path[-]
                    (a) edge node {$2$} (e)
                        edge node {$4$} (f)
                    (b) edge node {$2$} (a)
                    (e) edge node {$2$} (f);
                \path[-, red]
                    (c) edge node [red] {$4$} (b)
                        edge node [red] {$4$} (g)
                    (f) edge node [red] {$4$} (g)
                        edge node [red] {$2$} (b);

                \node[draw, circle, fill, inner sep=1.5pt, blue] (b1) at (6, 0) {};
                \node[draw, circle, fill, inner sep=1.5pt, blue] (c1) at (8, 0) {};
                \node[draw, circle, fill, inner sep=1.5pt] (d) at (10, 0) {};
                \node[draw, circle, fill, inner sep=1.5pt, blue] (f1) at (6, 2) {};
                \node[draw, circle, fill, inner sep=1.5pt, blue] (g1) at (8, 2) {};
                \node[draw, circle, fill, inner sep=1.5pt] (h) at (10, 2) {};
            
                \path[-]
                    (d) edge node {$3$} (c1)
                    (g1) edge node {$3$} (h)
                        edge node {$3$} (d)
                    (h) edge node {$3$} (d);

                \path[-, blue]
                    (c1) edge node [blue] {$4$} (b1)
                        edge node [blue] {$4$} (g1)
                    (f1) edge node [blue] {$4$} (g1)
                        edge node [blue] {$2$} (b1);

            \end{tikzpicture}
            \caption{\label{amalgam}An acylindrical visual splitting of $A_S$ as $A_X*_{A_Z}A_Y$, where $\Gamma_X$ is shown on the left, $\Gamma_Y$ is shown on the right and $\Gamma_Z$ is shown as a red and blue subgraph in $\Gamma_X$ and $\Gamma_Y$ respectively.}
        \end{figure}
        Let $A_S$ be the Artin group with presentation graph $\Gamma_S$ as shown in Figure~\ref{weirdgph}. The group $A_S$ is not spherical as $\Gamma_S$ is not complete, not $2$-dimensional as $\Gamma_S$ contains a $(2, 2, 4)$-triangle which corresponds to a spherical subgroup on more than two generators, and is not of FC-type as $\Gamma_S$ contains a $(3, 3, 3)$-triangle which corresponds to a \emph{non}-spherical subgroup on a clique. The group $A_S$ is not known to cocompactly cubulate, and indeed if the conjectural classification of cocompact cubulability of Artin groups \cite[Conjecture~B]{Haettel} holds then $A_S$ will not cocompactly cubulate as the $(3, 3, 3)$ triangle falls into the first bullet point. Finally, the group $A_S$ is not $2$-dimensional as stated above, so is not known to act properly and cocompactly on a $2$-complex satisfying the conditions of \cite[Theorems~A and~A.2]{OSAJDA2021107976}. Thus $A_S$ cannot be shown to satisfy the strong Tits alternative using previously known constructions.
        
        However, $A_S$ has a visual splitting $A_X*_{A_Z}A_Y$ as shown in Figure~\ref{amalgam} in which $A_X$ is $2$-dimensional and $(2, 2)$-free, so will have PIP and RP by \cite[Theorem~1.3]{Blufstein} and \cite[Theorem~3]{GODELLE200739} respectively, and $A_Y$ is even of FC-type so will have properties PIP and RP by \cite[Theorem~1.1]{Antolín_Foniqi_2022} and \cite[Theorem~0.3]{GodelleFC} respectively. This splitting is therefore acylindrical by Theorem~\ref{Thm} using the fact that $N_{\Gamma_X}(X\setminus Z)$ and $N_{\Gamma_Y}Y\setminus Z)$ are disjoint and $\Gamma_Z$ contains no odd labelled edges. Furthermore, $A_X$ is $2$-dimensional so satisfies the strong Tits alternative by \cite[Theorem~A]{MARTIN2024294}, and $A_Y$ is again FC-type, so satisfies the strong Tits alternative by \cite[Theorem~B]{MartinPrzytyckiFC}. Therefore Corollary~\ref{cor:titsArtin} implies that $A_S$ satisfies the Tits alternative.
        
        This example is by no means unique - indeed, if one removes the two vertical edges in $\Gamma_Z$ then this splitting would still satisfy the required conditions, there is a lot of freedom in the choice of the labels in this example, and other unrelated examples are easy to construct.
    \end{example}

\printbibliography
\end{document}